\documentclass{amsart}

\usepackage{amssymb,amsmath,latexsym,amsthm}
\usepackage{MnSymbol}
\usepackage{graphicx}
\usepackage{fontenc}%
\usepackage{marvosym}
\usepackage{paralist}
\usepackage{color}

\usepackage[verbose]{wrapfig}

\newtheorem{theorem}{Theorem}[section]

\newtheorem{lemma}[theorem]{Lemma}
\newtheorem{example}[theorem]{Example}

\usepackage{pstricks,pst-text,pst-grad,pst-node,pst-3dplot,pstricks-add,pst-poly}
\definecolor{pink}{rgb}{1, .75, .8}
\psset{arrows=->, labelsep=3pt, mnode=circle}

\definecolor{lgrey}{gray}{.85}

\newcommand{\norm}[1]{\left\|#1\right\|}  

\begin{document}

\title[Visibility in Proximal Delaunay Meshes]{Visibility in Proximal Delaunay Meshes}

\author[J.F. Peters]{J.F. Peters}
\email{James.Peters3@umanitoba.ca}
\address{Computational Intelligence Laboratory,
University of Manitoba, WPG, MB, R3T 5V6, Canada and Mathematics Department, Faculty of Arts and Science, Ad\i yaman University, 02040 Ad\i yaman, Turkey}
\thanks{The research has been supported by the Scientific and Technological Research
Council of Turkey (T\"{U}B\.{I}TAK) Scientific Human Resources Development
(BIDEB) under grants 2221-1059B211301223, 2221-1059B211402463 and Natural Sciences \&
Engineering Research Council of Canada (NSERC) discovery grant 185986.}


\date{}

\dedicatory{Dedicated to the Memory of Som Naimpally}

\begin{abstract}
This paper introduces a visibility relation $v$ (and the strong visibility relation $\mathop{v}\limits^{\doublewedge}$) on proximal Delaunay meshes.  A main result in this paper is that the visibility relation $v$ is equivalent to Wallman proximity.    In addition, a Delaunay triangulation region endowed with the visibility relation $v$ has a local Leader uniform topology.   
\end{abstract}


\maketitle

\section{Introduction}
Delaunay triangulations, introduced by B.N Delone [Delaunay]~\cite{Delaunay1934}, represent pieces of a continuous space.  A \emph{triangulation} is a collection of triangles, including the edges and vertices of the triangles in the collection.  

\setlength{\intextsep}{0pt}
\begin{wrapfigure}[8]{R}{0.40\textwidth}
\begin{minipage}{4.2 cm}
\begin{center}
 \begin{pspicture}
 (0.0,1.0)(2.5,3.8)
\psline[linestyle=solid,showpoints=true,dotscale=1.0]{-}(0.84,2.35)(2.22,1.35)(2.75,3.00)(2.95,3.50)(3.75,3.33)(0.84,2.35)
\rput(0.58,2.35){\footnotesize$\boldsymbol{p}$}
\rput(2.32,1.20){\footnotesize$\boldsymbol{q}$}
\rput(2.91,2.85){\footnotesize$\boldsymbol{r}$}
\rput(3.75,3.13){\footnotesize$\boldsymbol{s}$}
\rput(2.95,3.70){\footnotesize$\boldsymbol{t}$}
 \end{pspicture}
\caption{Visibility}
\label{fig:visiblePoint}
\end{center}
\end{minipage}
\end{wrapfigure}
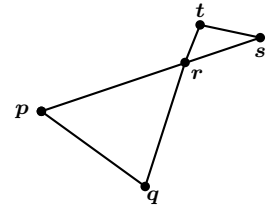
\setlength{\intextsep}{2pt}

A 2D \emph{Delaunay triangulation} of a set of sites (generators) $S\subset \mathbb{R}^2$ is a triangulation of the points in $S$.  The set of vertices (called sites) in a Delaunay triangulation define a \emph{Delaunay mesh}.  A Delaunay mesh endowed with a nonempty set of proximity relations is a \emph{proximal Delaunay mesh}.  A proximal Delaunay mesh is an example of a proximal relator space~\cite{Peters2014FilomatRelator}, which is an extension a Sz\'{a}z relator space~\cite{Szaz1987,Szaz2000,Szaz2009}.

Let $S\subset \mathbb{R}^2$ be a set of distinguished points called \emph{sites}, $p,q\in S$, $\overline{pq}$ straight line segment in the Euclidean plane.  A site $p$ in a line is \emph{visible} to another site $q$ in the same straight line segment, provided there is no other site between $p$ and $q$.  

\begin{example} {\bf Visible Points}.\\
A pair of Delaunay triangles $\bigtriangleup(pqr), \bigtriangleup(rst)$ are shown in Fig.~\ref{fig:visiblePoint}. Points $r,q$ is visible from $p$ but in the straight line segment $\overline{ps}$, $s$ is not visible from $p$.  Similarly, points  $p,r$ are visible from $q$ but in the straight line segment $\overline{qs}$, $t$ is not visible from $q$.  From $r$, points $p,q,s,t$ are visible.  \qquad \textcolor{black}{$\blacksquare$}
\end{example}

A straight edge connecting $p$ and $q$ is a \emph{Delaunay edge} if and only if the Vorono\"{i} region of $p$~\cite{Edelsbrunner2014,Peters2014arXivVoronoi} and Vorono\"{i} region of $q$ intersect along a common line segment~\cite[\S I.1, p. 3]{Edelsbrunner2001}.  For example, in Fig.~\ref{fig:convexPolygon}, the intersection of Vorono\"i regions $V_p, V_q$ is a triangle edge, {\em i.e.}, $V_p \cap V_q = \overline{xy}$.  Hence, $\overline{pq}$ is a Delaunay edge in Fig.~\ref{fig:convexPolygon}. 

\begin{figure}[!ht]
\begin{center}
 \begin{pspicture}
 (0.0,1.2)(5.5,4.0)
\psframe[linecolor=white](2.5,0.5)(6.5,3.5)
\rput(4.5,2){\PstPentagon[yunit=0.8,fillstyle=solid,fillcolor=lgrey,linestyle=solid,linecolor=gray]}
\pscircle[fillstyle=solid,fillcolor=black](4.5,2.0){0.08}
\psline[linestyle=solid,dotsep=0.03]{-}(4.5,2.0)(5.35,3.25)
\pscircle[fillstyle=solid,fillcolor=black](5.35,3.25){0.08}
\psline[linestyle=solid,dotsep=0.03]{-}(4.5,2.0)(5.95,1.55)
\pscircle[fillstyle=solid,fillcolor=black](5.95,1.55){0.08}
\psline[linestyle=solid]{-}(5.35,3.25)(5.95,1.55)
\psline[linewidth=0.8pt,linestyle=dotted,dotsep=0.05,linecolor=black]{-}(4.53,2.8)(4.53,3.5)
\pscircle[fillstyle=solid,fillcolor=white,linecolor=black](4.53,2.8){0.08}
\psline[linewidth=0.8pt,linestyle=dotted,dotsep=0.05,linecolor=black]{-}(4.53,3.5)(5.00,3.8)
\pscircle[fillstyle=solid,fillcolor=white,linecolor=black](4.53,3.5){0.08}
\psline[linewidth=0.8pt,linestyle=dotted,dotsep=0.05,linecolor=black]{-}(5.00,3.8)(6.80,3.5)
\pscircle[fillstyle=solid,fillcolor=white,linecolor=black](5.00,3.8){0.08}
\psline[linewidth=0.8pt,linestyle=dotted,dotsep=0.05,linecolor=black]{-}(6.80,3.5)(6.48,2.36)
\pscircle[fillstyle=solid,fillcolor=white,linecolor=black](6.80,3.5){0.08}
\psline[linewidth=0.8pt,linestyle=dotted,dotsep=0.05,linecolor=black]{-}(5.45,2.30)(6.48,2.36)
\pscircle[fillstyle=solid,fillcolor=white,linecolor=black](6.48,2.36){0.08}
\pscircle[fillstyle=solid,fillcolor=white,linecolor=black](5.45,2.30){0.08}
\psline[linewidth=0.8pt,linestyle=dotted,dotsep=0.05,linecolor=black]{-}(3.62,2.20)(2.23,2.5)
\pscircle[fillstyle=solid,fillcolor=white,linecolor=black](3.62,2.20){0.08}
\psline[linewidth=0.8pt,linestyle=dotted,dotsep=0.05,linecolor=black]{-}(3.88,1.35)(3.28,0.75)
\pscircle[fillstyle=solid,fillcolor=white,linecolor=black](3.88,1.35){0.08}
\psline[linewidth=0.8pt,linestyle=dotted,dotsep=0.05,linecolor=black]{-}(5.00,1.38)(5.38,0.75)
\psline[linewidth=0.8pt,linestyle=dotted,dotsep=0.05,linecolor=black]{-}(5.38,0.75)(6.88,1.25)
\psline[linewidth=0.8pt,linestyle=dotted,dotsep=0.05,linecolor=black]{-}(6.88,1.25)(6.48,2.36)
\pscircle[fillstyle=solid,fillcolor=white,linecolor=black](5.00,1.38){0.08}
\pscircle[fillstyle=solid,fillcolor=white,linecolor=black](5.38,0.75){0.08}
\pscircle[fillstyle=solid,fillcolor=white,linecolor=black](6.88,1.25){0.08}
\pscircle[fillstyle=solid,fillcolor=white,linecolor=black](6.48,2.36){0.08}
\rput(4.45,2.2){\footnotesize$\boldsymbol{p}$}\rput(6.15,1.55){\footnotesize$\boldsymbol{q}$}
\rput(5.55,3.25){\footnotesize$\boldsymbol{r}$}\rput(5.45,2.50){\footnotesize$\boldsymbol{x}$}
\rput(4.95,1.18){\footnotesize$\boldsymbol{y}$}
\rput(5.98,2.85){\tiny$\boldsymbol{\bigtriangleup(pqr)}$}
\rput(3.95,2.1){\footnotesize$\boldsymbol{V_p}$}
\rput(5.50,1.10){\footnotesize$\boldsymbol{V_q}$}
 \end{pspicture}
\end{center}
\caption[]{Delaunay triangle \footnotesize $\boldsymbol{\bigtriangleup(pqr)}$}
\label{fig:convexPolygon}
\end{figure}
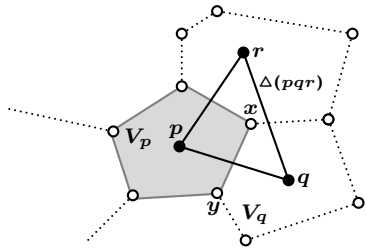  

A triangle with vertices $p,q,r\in S$ is a \emph{Delaunay triangle} (denoted $\bigtriangleup(pqr)$ in Fig.~\ref{fig:convexPolygon}), provided the edges in the triangle are Delaunay edges.  
This paper introduces proximal Delaunay triangulation regions derived from the sites of Vorono\"{i} regions~\cite{Peters2014arXivVoronoi}, which are named after the Ukrainian mathematician Georgy Vorono\"{i}~\cite{Voronoi1907}.

A nonempty set $A$ of a space $X$ is a \emph{convex set}, provided $\alpha A + (1-\alpha)A\subset A$ for each $\alpha\in[0,1]$~\cite[\S 1.1, p. 4]{Beer1993}.  
A \emph{simple convex set} is a closed half plane (all points on or on one side of a line in $R^2$~\cite{Edelsbrunner2014}). The edges in a Delaunay mesh are examples of convex sets.  
A closed set $S$ in the Euclidean space $E^n$ is \emph{convex} if and only if to each point in $E^n$ there corresponds a unique nearest point in $S$.   
For $z\in S$, a closed set in $\mathbb{R}^n, S_z = \left\{x\in E: \left\|x - z\right\| = \mathop{inf}\limits_{y\in S}\left\|x - y\right\|\right\}$ is a \emph{convex cone}~\cite{Phelps1957}.

\begin{lemma}\label{lem:convexity}~{\rm \cite[\S 2.1, p. 9]{Edelsbrunner2014}} The intersection of convex sets is convex.
\end{lemma}
\begin{proof}
Let $A,B\subset \mathbb{R}^2$ be convex sets and let $K = A\cap B$.  For every pair points $x,y\in K$, the line segment $\overline{xy}$ connecting $x$ and $y$ belongs to $K$, since this property holds for all points in $A$ and $B$.   Hence, $K$ is convex.
\end{proof}

\section{Preliminaries}
Delaunay triangles are defined on a finite-dimensional normed linear space $E$ that is topological.  For simplicity, $E$ is the Euclidean space $\mathbb{R}^2$.  The \emph{closure} of $A\subset E$ (denoted $\mbox{cl}A$) is defined by
\begin{align*}
\mbox{cl}(A) &= \left\{x\in X: D(x,A)=0\right\},\ \mbox{where}\\
D(x,A) &= inf\left\{\norm{x-a}: a\in A\right\},
\end{align*}
{\em i.e.}, $\mbox{cl}(A)$ is the set of all points $x$ in $X$ that are close to $A$ ($D(x,A)$ is the Hausdorff distance \cite[\S 22, p. 128]{Hausdorff1914} between $x$ and the set $A$ and $\norm{x,a}$ is the Euclidean distance between $x$ and $a$).

\setlength{\intextsep}{0pt}
\begin{wrapfigure}[11]{R}{0.40\textwidth}
\begin{minipage}{4.2 cm}
\begin{center}
 \begin{pspicture}
 (0.0,0.0)(2.5,3.8)
\pscircle[fillstyle=solid,fillcolor=lgrey,linestyle=dotted](1.0,1.4){1.78}
\pscircle[fillstyle=solid,fillcolor=lgrey,linestyle=dotted](3.25,3.25){0.80}
\pscircle[linestyle=dashed,dash=3pt 1.5pt](2.7,2.7){0.75}
\psline[linestyle=solid,showpoints=true,dotscale=1.0]{-}(0.84,2.35)(2.22,1.35)(2.75,3.00)(2.95,3.50)(3.75,3.33)(0.84,2.35)
\psline[linestyle=solid,showpoints=true,dotscale=1.0]{-}(0.84,2.35)(2.22,1.35)(2.22,0.84)(0.84,2.35)
\psline[linestyle=solid,showpoints=true,dotscale=1.0]{-}(0.50,1.35)(2.22,0.84)
\psline[linestyle=solid,showpoints=true,dotscale=1.0]{-}(0.50,1.35)(0.84,2.35)
\psline[linestyle=solid,showpoints=true,dotscale=1.0]{-}(0.50,1.35)(0.00,1.00)(0.00,2.00)(0.50,1.35)
\psline[linestyle=solid,showpoints=true,dotscale=1.0]{-}(2.45,2.00)(2.05,2.75)
\psline[linestyle=solid,showpoints=true,dotscale=1.0]{-}(2.45,2.00)(2.05,2.75)
\psline[linestyle=solid,showpoints=true,dotscale=1.0]{-}(2.45,2.00)(0.84,2.35)
\psline[linestyle=solid,showpoints=true,dotscale=1.0]{-}(0.84,2.35)(0.00,2.00)
\psline[linestyle=solid,showpoints=true,dotscale=1.0]{-}(0.00,1.00)(2.22,0.84)
\rput(0.88,2.55){\footnotesize$\boldsymbol{p}$}\rput(1.88,2.85){\footnotesize$\boldsymbol{w}$}
\rput(2.45,0.80){\footnotesize$\boldsymbol{q}$}\rput(2.55,1.85){\footnotesize$\boldsymbol{x}$}
\rput(0.54,1.15){\footnotesize$\boldsymbol{u}$}\rput(0.94,1.55){\footnotesize$\boldsymbol{C}$}
\rput(2.91,2.85){\footnotesize$\boldsymbol{r}$}\rput(0.84,0.35){\footnotesize$\boldsymbol{B}$}
\rput(3.75,3.13){\footnotesize$\boldsymbol{s}$}\rput(3.55,2.90){\footnotesize$\boldsymbol{A}$}
\rput(2.75,2.25){\footnotesize$\boldsymbol{D}$}
\rput(2.95,3.70){\footnotesize$\boldsymbol{t}$}
 \end{pspicture}
\caption{Far}
\label{fig:stronglyFar}
\end{center}
\end{minipage}
\end{wrapfigure}
\setlength{\intextsep}{2pt}

Let $A^c$ denote the complement of $A$ (all points of $E$ not in $A$).  The \emph{boundary} of $A$ (denoted $\mbox{bdy}A$) is the set of all points that are near $A$ and near $A^c$~\cite[\S 2.7, p. 62]{Naimpally2013}.  An important structure is the \emph{interior} of $A$ (denoted $\mbox{int}A$), defined by $\mbox{int}A = \mbox{cl}A - \mbox{bdy}A$.  For example, the interior of a Delaunay edge $\overline{pq}$ are all of the points in the segment, except the endpoints $p$ and $q$.
  
In general, a \emph{relator} is a nonvoid family of relations $\mathcal{R}$ on a nonempty set $X$.  The pair $(X,\mathcal{R})$ is called a relator space.  Let $E$ be endowed with the proximal relator
\[ 
\mathcal{R}_{\delta} = \left\{\delta, \mathop{\delta}\limits^{\doublewedge}, \underline{\delta}, \stackrel{\text{\normalsize$\delta$}}{\text{\tiny$\doublevee$}}\right\}\ \mbox{(Proximal Relator, cf.~\cite{Peters2014FilomatRelator})}.
\]
The Delaunay tessellated space $E$ endowed with the proximal relator $\mathcal{R}_{\delta}$ (briefly, $\mathcal{R}$) is a \emph{Delaunay proximal relator space}.
 
The proximity relations $\delta$ (near), $\mathop{\delta}\limits^{\doublewedge}$ (strongly near) and their counterparts $\underline{\delta}$ (far) and $\stackrel{\text{\normalsize$\delta$}}{\text{\tiny$\doublevee$}}$ (strongly far) facilitate the description of properties of Delaunay edges, triangles, triangulations and regions.  Let $A,B\subset E$.  The set $A$ is near $B$ (denoted $A\ \delta\ B$), provided $\mbox{cl}A\cap \mbox{cl}B\neq \emptyset$~\cite{Concilio2009}.  The Wallman proximity $\delta$ (named after H. Wallman~\cite{Wallman1938}) satisfies the four \u{C}ech proximity axioms~\cite[\S 2.5, p. 439]{Cech1966} and is central in near set theory~\cite{Naimpally2013,Peters2012ams}.  Sets $A,B$ are \emph{far} apart (denoted $A\ \underline{\delta}\ B$), provided $\mbox{cl}A\cap \mbox{cl}B = \emptyset$. For example, Delaunay edges $\overline{pq}\ \delta\ \overline{qr}$ are near, since the edges have a common point, {\em i.e.}, $q\in \overline{pq}\cap \overline{qr}$ (see, {\em e.g.}, $\overline{pq}\ \delta\ \overline{qr}$ in Fig.~\ref{fig:convexPolygon}).  By contrast, edges $\overline{pr},\overline{xy}$ have no points in common in Fig.~\ref{fig:convexPolygon}, {\em i.e.}, $\overline{pr}\ \underline{\delta}\ \overline{xy}$.

Vorono\"{i} regions $V_p,V_q$ are strongly near (denoted $V_p\ \mathop{\delta}\limits^{\doublewedge}\ V_q$) if and only if the regions have a common edge.  For example, $V_p\ \mathop{\delta}\limits^{\doublewedge}\ V_q$ in Fig.~\ref{fig:convexPolygon}.  \emph{Strongly near} Delaunay triangles have a common edge.  Delaunay triangles $\bigtriangleup(pqr)$ and $\bigtriangleup(qrt)$ are strongly near in Fig.~\ref{fig:nearTriangles}, since edge $\overline{qr}$ is common to both triangles.  In that case, we write $\bigtriangleup(pqr)\ \mathop{\delta}\limits^{\doublewedge}\ \bigtriangleup(qrt)$.

Nonempty sets $A \stackrel{\text{\normalsize$\delta$}}{\text{\tiny$\doublevee$}} C$ are strongly far apart (denoted $A,C$), provided $C\subset \mbox{int}(\mbox{cl}B)$ and $A\ \underline{\delta}\ B$.  

\begin{example} {\bf Far and Strongly Far Sets}.\\
In the Delaunay mesh in Fig.~\ref{fig:stronglyFar}, sets $A$ and $B$ have no points in common.  Hence, $A\ \underline{\delta}\ B$ ($A$ is far from $B$.  Also in Fig.~\ref{fig:stronglyFar}, let $C = \left\{\bigtriangleup(pqu)\right\}$.  Consequently, $C\subset \mbox{int}(\mbox{cl}B)$, such that triangle $\bigtriangleup(pqu)$ lies in the interior of the closure of $B$.  Hence, $A \stackrel{\text{\normalsize$\delta$}}{\text{\tiny$\doublevee$}} C$.
  \qquad \textcolor{black}{$\blacksquare$}
\end{example}

\begin{figure}[!ht]
\begin{center}
 \begin{pspicture}
 (0.0,0.5)(7.5,3.8)
\pscircle[fillstyle=solid,fillcolor=lgrey,linecolor=black,linestyle=solid](5.30,2.05){1.58}
\pscircle[linecolor=black,linestyle=dotted,dotsep=0.05](3.88,2.35){1.38}
\pscircle[fillstyle=solid,fillcolor=white,linecolor=black](3.88,2.35){0.08}
\psline[linestyle=solid,showpoints=true,dotscale=1.0]{-}(2.84,2.35)(4.22,1.35)(4.75,3.00)(2.84,2.35)
\psline[linestyle=dotted,dotsep=0.03]{-}(3.88,2.35)(2.58,1.25)
\psline[linestyle=dotted,dotsep=0.03]{-}(3.88,2.35)(5.50,2.00)
\psline[linestyle=dotted,dotsep=0.03]{-}(3.88,2.35)(3.50,3.80)
\pscircle[fillstyle=solid,fillcolor=black,linecolor=black](6.50,1.80){0.07}
\pscircle[fillstyle=solid,fillcolor=white,linecolor=black](5.50,2.00){0.08}
\psline[linestyle=solid]{-}(4.75,3.00)(6.50,1.80)
\psline[linestyle=solid]{-}(4.22,1.35)(6.50,1.80)
\rput(2.58,2.35){\footnotesize$\boldsymbol{p}$}\rput(3.08,2.85){\normalsize$\boldsymbol{B}$}
\rput(4.32,1.15){\footnotesize$\boldsymbol{q}$}
\rput(4.91,3.15){\footnotesize$\boldsymbol{r}$}
\rput(6.69,1.93){\footnotesize$\boldsymbol{t}$}\rput(5.69,2.93){\normalsize$\boldsymbol{A}$}
 \end{pspicture}
\end{center}
\caption[]{Strongly Visible Sets $A\ \mathop{v}\limits^{\doublewedge}\  B$}
\label{fig:nearTriangles}
\end{figure}

Let $A,B$ be subsets in a Delaunay mesh, $\bigtriangleup(pqr)\in B, \bigtriangleup(qrt)\in A$. Subsets $A,B$ in a Delaunay mesh are \emph{visible} to each other (denoted $A v B$), provided at least one triangle vertex $x \in \mbox{cl}A\cap\mbox{cl}B$.  $A,B$ are \emph{strongly visible} to each other (denoted $A \mathop{v}\limits^{\doublewedge}  B$), provided at least one triangle edge is common to $A$ and $B$.

\begin{example} {\bf Visibility in Delaunay Meshes}.\\
In the Delaunay mesh in Fig.~\ref{fig:stronglyFar}, $A\ v\ D$, since $A$ and $D$ have one triangle vertex is common, namely, vertex $r$.  Sets $B$ and $D$ in Fig.~\ref{fig:stronglyFar} are strongly visible ({\em i.e.}, $B \mathop{v}\limits^{\doublewedge} D$), since edge $\overline{wx }$ is common to $B$ and $D$.  In Fig.~\ref{fig:stronglyFar}, let $C = \left\{\bigtriangleup(pqu)\right\}$.  Then $C \mathop{v}\limits^{\doublewedge} B$, since $C\subset B$.  In Fig.~\ref{fig:nearTriangles}, edge $\overline{qr}$ is common to $A$ and $B$.  $\overline{qr}$ is visible from $p\in B$ and from $t\in A$.  Hence, $A \mathop{v}\limits^{\doublewedge}  B$ \qquad \textcolor{black}{$\blacksquare$}
\end{example}

Subsets $A,B$ in a Delaunay mesh are \emph{invisible} to each other (denoted $A\ \underline{v}\ B$), provided $\mbox{cl}A\cap\mbox{cl}B = \emptyset$, {\em i.e.}, $A$ and $B$ have no triangle vertices in common.  $A,B$ are \emph{strongly invisible} to each other (denoted $A \stackrel{\text{\normalsize$v$}}{\text{\tiny$\doublevee$}}  B$), provided $C\ \underline{v}\ A$ for all sets of mesh triangles $C\subset B$.

\begin{example} {\bf Invisible and Strongly Invisible Subsets in a Delaunay Mesh}.\\
In the Delaunay mesh in Fig.~\ref{fig:stronglyFar}, $A$ and $B$ are not visible to each other, since $\mbox{cl}A\cap\mbox{cl}B = \emptyset$, {\em i.e.}, $A$ and $B$ have no triangle vertices in common.  In Fig.~\ref{fig:stronglyFar}, let $C = \left\{\bigtriangleup(pqu)\right\}$.  Then $A \stackrel{\text{\normalsize$v$}}{\text{\tiny$\doublevee$}}  B$ ($A$ and $B$ are strongly invisible to each other), since $C\ \underline{v}\ A$ for all sets of mesh triangles $C\subset B$.
\qquad \textcolor{black}{$\blacksquare$}
\end{example}

\section{Main Results}
The Delaunay visibility relation $v$ is equivalent to the proximity $\delta$. 

\begin{lemma}\label{lem:Wallman}
Let $A,B$ be subsets in a Delaunay mesh.  $A\ \delta\ B$ if and only if $A\ v\ B$.
\end{lemma}
\begin{proof}
$A\ \delta\ B\Leftrightarrow \mbox{cl}A\cap\mbox{cl}B\neq \emptyset\Leftrightarrow A$ and $B$ have a triangle vertex in common if and only if $A\ v\ B$.
\end{proof}

\begin{theorem}\label{thm:Wallman} The visibility relation $v$ is a Wallman proximity.
\end{theorem}
\begin{proof}
Immediate from Lemma~\ref{lem:Wallman}.
\end{proof}

\begin{lemma}\label{lem:strongWallman}
Let $A,B$ be subsets in a Delaunay mesh. $A \mathop{v}\limits^{\doublewedge}  B$ if and only if $A v B$.
\end{lemma}
\begin{proof} $A \mathop{v}\limits^{\doublewedge}  B \Leftrightarrow \overline{pq}$ for some triangle edge common to $A$ and $B \Leftrightarrow A v B$, since $\overline{pq}$ is visible from a vertex in $A$ and from a vertex in $B$ and $A$ and $B$ have vertices in common.
\end{proof}

\begin{theorem} The strong visibility relation $\mathop{v}\limits^{\doublewedge}$ is a Wallman proximity.
\end{theorem}
\begin{proof}
Immediate from Theorem~\ref{thm:Wallman} and Lemma~\ref{lem:strongWallman}.
\end{proof}

\begin{theorem}
Let $A,B$ be subsets in a Delaunay mesh.  Then
\begin{compactenum}[1$^o$]
\item $A \stackrel{\text{\normalsize$v$}}{\text{\tiny$\doublevee$}}  B$ implies $A\ \underline{v}\ B$.
\item $A \stackrel{\text{\normalsize$v$}}{\text{\tiny$\doublevee$}}  B$ if and only if $A\ \stackrel{\text{\normalsize$\delta$}}{\text{\tiny$\doublevee$}}\ B$.
\end{compactenum}
\end{theorem}
\begin{proof}$\mbox{}$\\ 
1$^o$: Given $A \stackrel{\text{\normalsize$v$}}{\text{\tiny$\doublevee$}}  B$, then $A$ and $B$ have no triangle vertices in common.  Hence, $A\ \underline{v}\ B$.\\
2$^o$: $A \stackrel{\text{\normalsize$v$}}{\text{\tiny$\doublevee$}}  B$ if and only if $A$ and $B$ have no triangles in common if and only if $A\ \stackrel{\text{\normalsize$\delta$}}{\text{\tiny$\doublevee$}}\ B$.
\end{proof}

\noindent Theorem~\ref{thm:visiblyConvex} is an extension of Theorem 3.1 in~\cite{Peters2014arXivDelaunay}, which results from Theorem~\ref{thm:Wallman}.

\begin{theorem}\label{thm:visiblyConvex}
The following statements are equivalent.
\begin{compactenum}[1$^o$]
\item $\bigtriangleup(pqr)$ is a Delaunay triangle.
\item Circumcircle $\bigcirc(pqr)$ has center $u = \mbox{cl}V_p\cap \mbox{cl}V_q\cap \mbox{cl}V_r$.
\item $V_p\ \mathop{v}\limits^{\doublewedge}\ V_q\ \mathop{v}\limits^{\doublewedge}\  V_r$.
\item $\bigtriangleup(pqr)$ is the union of convex sets.
\end{compactenum}
\end{theorem}

Let $P$ be a polygon.  Two points $p,q\in P$ are \emph{visible}, provided the line segment $\overline{pq}$ is in $\mbox{int}P$~\cite{Ghosh2007}.  Let $p,q\in S$, $L$ a finite set of straight line segments and let $\overline{pq}\in L$.  Points $p,q$ are visible from each other, which implies that $\overline{pq}$ contains no point of $S-\left\{p,q\right\}$ in its interior and $\overline{pq}$ shares no interior point with a constraining line segment in $L-\overline{pq}$.  That is, $\mbox{int}\overline{pq}\cap S = \emptyset$ and $\overline{pq}\cap \overline{xy} = \emptyset$ for all $\overline{xy}\in L$~\cite[\S II, p. 32]{Edelsbrunner2001}

\begin{theorem}
If points in $\mbox{int}\ \overline{pq}$ are visible from $p,q$, then $\mbox{int}\ \overline{pq}\ \underline{v}\ S-\left\{p,q\right\}$ and $\overline{pq}\ \underline{v}\ \overline{xy}\in L-\overline{pq}$ for all $x,y\in S-\left\{p,q\right\}$.
\end{theorem}
\begin{proof}
Symmetric with the proof of Theorem 3.2~\cite{Peters2014arXivDelaunay}.
\end{proof}

A \emph{Delaunay triangulation region} $\mathcal{D}$ is a collection of Delaunay triangles such that every pair triangles in the collection is strongly near.  That is, every Delaunay triangulation region is a triangulation of a finite set of sites and the triangles in each region are pairwise strongly near.  \emph{Proximal Delaunay triangulation regions} have at least one vertex in common.  From Lemma~\ref{lem:convexity} and the definition of a Delaunay triangulation region, observe

\begin{lemma}\label{lem:convexPolygon}\cite{Peters2014arXivDelaunay}
A Delaunay triangulation region is a convex polygon.  
\end{lemma}

\begin{theorem}\cite{Peters2014arXivDelaunay}
Proximal Delaunay triangulation regions are convex polygons.
\end{theorem}

A \emph{local Leader uniform topology}~\cite{Leader1959} on a set in the plane is determined by finding those sets that are close to each given set.

\begin{theorem}\label{thm:Leader}\cite{Peters2014arXivDelaunay}
Every Delaunay triangulation region has a local Leader uniform topology {\rm (application of~{\rm \cite{Leader1959}})}.
\end{theorem}

\begin{theorem}\cite{Peters2014arXivDelaunay}
A Delaunay triangulation region endowed with the visibility relation $v$ has a local Leader uniform topology.
\end{theorem}
\begin{proof} Let $\mathcal{D}$ be a Delaunay triangulation region. From Theorem~\ref{thm:Wallman} and Theorem~\ref{thm:Leader}, determine all subsets of $\mathcal{D}$ that are visible from each given subset of $\mathcal{D}$.
For each $A\subset \mathcal{D}$, this procedure determines a family of Delaunay triangles that are visible from (near) each $A$. By definition, this procedure induces a local Leader uniform topology on $\mathcal{D}$.
\end{proof}


\end{document}